\documentclass[11pt, final]{article}
\usepackage{a4}
\usepackage{amsmath}%
\usepackage{amstext}%
\usepackage{amssymb}%
\usepackage{showkeys}%
\usepackage{epsfig}%
\usepackage{cite}
\usepackage{tikz}
\usepackage{subfig}
\usepackage{caption}
\usepackage{multirow}
\usepackage{booktabs}
\usepackage{floatrow}

\usepackage{listings}
\newtheorem{theorem}{Theorem}

\newtheorem{algorithm}[theorem]{Algorithm}

\newtheorem{lemma}[theorem]{Lemma}

\newtheorem{proposition}[theorem]{Proposition}
\newtheorem{rem}[theorem]{Remark}
\newenvironment{remark}{\rem\rm}{\endrem}

\newcounter{unnumber}

\newtheorem{prf}[unnumber]{Proof.}
\newenvironment{proof}{\prf\rm}{\hfill{$\blacksquare$}\endprf}
\newcommand{\N}{\mathbb{N}}%
\newcommand{\ol}{\overline}%
\newcommand{\ul}{\underline}%

\DeclareMathOperator*\dom{dom}%
\DeclareMathOperator*\gr{Gr}%
\DeclareMathOperator*\ran{ran}%
\DeclareMathOperator*\id{Id}%

\DeclareMathOperator*\zer{zer}

\title{A hybrid proximal-extragradient algorithm with inertial effects}

\author{Radu Ioan Bo\c{t} \thanks{University of Vienna, Faculty of Mathematics, Oskar-Morgenstern-Platz 1, A-1090 Vienna, Austria,
email: radu.bot@univie.ac.at. Research partially supported by DFG (German Research Foundation), project BO 2516/4-1.} \and
Ern\"{o} Robert Csetnek \thanks {University of Vienna, Faculty of Mathematics, Oskar-Morgenstern-Platz 1, A-1090 Vienna, Austria,
email: ernoe.robert.csetnek@univie.ac.at. Research supported by DFG (German Research Foundation), project BO 2516/4-1.}}

\begin{document}
\maketitle

\noindent \textbf{Abstract.} We incorporate inertial terms in the hybrid proximal-extragradient algorithm and investigate the convergence 
properties of the resulting iterative scheme designed for finding the zeros of a maximally monotone operator in real Hilbert spaces. The convergence analysis
relies on extended Fej\'{e}r monotonicity techniques combined with the celebrated Opial Lemma. We also 
show that the classical hybrid proximal-extragradient algorithm and the inertial versions of the proximal point, the forward-backward and 
the forward-backward-forward algorithms can be embedded in the framework of the proposed iterative scheme. 
\vspace{0.001cm}

 \vspace{1ex}


\section{Introduction}\label{sec-intr}

The problem of numerically approaching the set of zeros of a maximally monotone operator in real Hilbert spaces 
is a topic of relevance for research communities working in different mathematical areas, like
partial differential equations, evolution systems and convex optimization, with wide applications to real-life problems as in image processing, signal recovery, 
classification via machine learning, location theory, average consensus in network coloring, clustering, etc.

The classical iterative scheme for solving this problem is the proximal point algorithm (see \cite{rock-prox}): $$x^{k+1}=(\id +c_kT)^{-1}(x^k) \quad \forall k \geq 0,$$ 
where $T:{\cal H}\rightrightarrows {\cal H}$ is a maximally monotone operator, $\cal H$ is a real Hilbert space, 
$\id$ is the identity operator on $\cal H$ and $(c_k)_{k\in\N}$ is a real sequence fulfilling $\liminf_{k\rightarrow+\infty} c_k>0$. 
In case $\zer T=\{x\in{\cal H}:0\in Tx\}\neq\emptyset$, the above algorithm weakly converges to a point in $\zer T$, no matter how the staring point $x^0\in{\cal H}$ is chosen. 

Following the ideas from \cite{bu-iu-sv1997}, developed in the context of dealing with variational inequalities, and \cite{sol-sv1999-jca}, 
the following {\it hybrid proximal-extragradient algorithm} has been proposed in \cite{sol-sv1999}:
\begin{algorithm}\label{alg-hyb-classic} Choose $x^0\in{\cal H}$, $\sigma\in[0,1)$ and $(c_k)_{k\in\N}$ such that $c_k\geq \ul c>0$ for all $k\in\N$. 
For all $k\in\N$ consider the following iterative scheme: 
\begin{itemize} \item[(i)] for some 
$\varepsilon_k\geq 0$, choose $v^k\in T^{[\varepsilon_k]}(y^k)$ such that $$\|c_kv^k+y^k-x^k\|^2+2c_k\varepsilon_k\leq \sigma^2\|y^k-x^k\|^2;$$
\item[(ii)] define $x^{k+1}=x^k-c_kv^k$.\end{itemize}
\end{algorithm}

In the above iterative scheme,  $T^{[\varepsilon]}:{\cal H}\rightrightarrows {\cal H}$ denotes the $\varepsilon$-enlargement of the operator $T$. 
It is shown in \cite{sol-sv1999} that the sequence $(x^k)_{k\in\N}$ weakly converges to a point in $\zer T$, provided this set is nonempty. We refer the reader
to \cite{mo-sv2010}  for iteration complexity results and also to \cite{sv} for a more general treatment of the hybrid-type proximal-extragradient 
methods. Several classical algorithms from the literature, like the classical proximal point, the forward-backward and the forward-backward-forward 
algorithms can derived as particular instances from the hybrid proximal-extragradient iterative scheme. Let us notice that the 
{\it forward-backward} and the {\it forward-backward-forward} (see \cite{tseng}) algorithms are designed for finding the zeros of the sum of 
two maximally monotone operators, one of them being single-valued, their formulations depending whether the single-valued operator is cocoercive 
or (only) monotone and Lipschitz continuous. The book \cite{bauschke-book} is an excellent reference for anyone interested in 
proximal algorithms. 

In this paper we will focus on the class of so-called \textit{inertial proximal methods}, the origins of which go back to \cite{alvarez2000, alvarez-attouch2001}.
The idea behind the iterative scheme relies on the use of an implicit discretization of a differential system of second-order in time and it was employed for the first time in  the context of finding the zeros of a
maximally monotone operator in \cite{alvarez-attouch2001}. One  of the main features of the inertial proximal algorithm is that the next iterate is defined by making use of the previous two iterates. It also turns out that
the method is a generalization of the classical proximal point one (see \cite{rock-prox}). Since its introduction, one can notice an increasing interest in the class of inertial type algorithms, see \cite{alvarez2000, alvarez-attouch2001,
cabot-frankel2011, mainge2008, mainge-moudafi2008, moudafi-oliny2003, b-c-h-inertial, b-c-inertial, b-c-inertial-admm, att-peyp-red, pesq-pust}. Especially noticeable is that these ideas where also used in \cite{moudafi-oliny2003} 
in the context of determining the zeros of the sum of a maximally monotone operator and a (single-valued) cocoercive operator, giving rise to the so-called inertial forward-backward algorithm. 
We also notice that an inertial forward-backward-forward algorithm  has been proposed in \cite{b-c-inertial} for the same problem in case the single-valued operator is monotone and Lipschitz continuous.

In this note we propose a hybrid proximal-extragradient algorithm with inertial and memory effects. The convergence 
of the iterative scheme relies on extended Fej\'{e}r monotonicity techniques adapted to the needs of the inertial-type numerical scheme. 
Moreover, we also show, like in \cite{sv}, that the classical hybrid proximal-extragradient algorithm, the inertial proximal point algorithm and
the inertial versions of the forward-backward and forward-backward-forward algorithms can be derived from the inertial hybrid proximal-extragradient scheme proposed in the paper.

\section{Preliminaries}\label{sec-pr}

In this section we recall some notations and results in order to make the paper self contained. 
For the notions and results presented as follows we refer the reader to \cite{bo-van, b-hab, bauschke-book, simons, bur-iu-book, teza-csetnek}. 
Let $\N= \{0,1,2,...\}$ be the set of nonnegative integers.
Let ${\cal H}$ be a real Hilbert space with \textit{inner product} $\langle\cdot,\cdot\rangle$ and associated \textit{norm} $\|\cdot\|=\sqrt{\langle \cdot,\cdot\rangle}$.
The symbols $\rightharpoonup$ and $\rightarrow$ denote weak and strong convergence, respectively.

For an arbitrary set-valued operator $T:{\cal H}\rightrightarrows {\cal H}$ we denote by $\gr T=\{(x,u)\in {\cal H}\times {\cal H}:u\in Tx\}$ its \emph{graph}, 
by $\dom T=\{x \in {\cal H} : Tx \neq \emptyset\}$
its \emph{domain}, by $\ran T=\cup_{x\in{\cal{H}}} Tx$ its {\it range}.
We use also the notation $\zer T=\{x\in{\cal{H}}:0\in Tx\}$ for the \emph{set of zeros} of $T$. We say that $T$ is \emph{monotone}
if $\langle x-y,u-v\rangle\geq 0$ for all $(x,u),(y,v)\in\gr T$. A monotone operator $T$ is said to be \emph{maximally monotone}, if there exists no proper 
monotone extension of the graph of $T$ on ${\cal H}\times {\cal H}$.
The \emph{resolvent} of $T$, $J_T:{\cal H} \rightrightarrows {\cal H}$, is defined by $p\in J_T x \ \mbox{if and only if} \ (p,x-p)\in\gr T$. 
Moreover, if $T$ is maximally monotone, then $J_T:{\cal H} \rightarrow {\cal H}$ is single-valued and maximally monotone
(see \cite[Proposition 23.7 and Corollary 23.10]{bauschke-book}). 

Let $\gamma>0$. A single-valued operator $A:{\cal H}\rightarrow {\cal H}$ is said to be \textit{$\gamma$-cocoercive} if $\langle x-y,Ax-Ay\rangle\geq \gamma\|Ax-Ay\|^2$ for all $(x,y)\in {\cal H}\times {\cal H}$.
Moreover, $A$ is \textit{$\gamma$-Lipschitzian} if $\|Ax-Ay\|\leq \gamma\|x-y\|$ for all $(x,y)\in {\cal H}\times {\cal H}$. 

Let $T:{\cal H}\rightrightarrows {\cal H}$ be a monotone operator and $\varepsilon\geq 0$. The $\varepsilon$-enlargement of $T$, denoted by 
$T^{[\varepsilon]}:{\cal H}\rightrightarrows {\cal H}$, is defined by
$$T^{[\varepsilon]}(x)=\{u\in{\cal H}:\langle x-y,u-v\rangle\geq -\varepsilon \ \forall (y,v)\in\gr T \}.$$ 

Introduced in \cite{bu-iu-sv1997}, this notion proved to possess fruitful properties in connection with the theory of monotone operators \cite{bur-iu-book, BCW-set-val, bur-sag-sv}, being 
used also in the formulation of several numerical schemes of proximal-type. The following properties, which will be used throughout the paper, have been taken from \cite{sv}.

\begin{proposition}\label{prop-enl} Let $T,T_1,T_2:{\cal H}\rightrightarrows {\cal H}$ be maximally monotone operators and 
$A:{\cal H}\rightarrow {\cal H}$ be $\gamma$-cocoercive, where $\gamma >0$. The following hold: 

(i) $T=T^{[0]}$; 

(ii) if $0\leq\varepsilon_1\leq\varepsilon_2$, then $T^{[\varepsilon_1]}(x)\subseteq T^{[\varepsilon_2]}(x)$ for all $x\in{\cal H}$; 

(iii) if $u_k\in T^{[\varepsilon_k]}(x_k)$ for all $k\in\N$, $x_k\rightharpoonup x$, $u_k\rightarrow u$ and $\varepsilon_k\rightarrow\varepsilon$, 
then $u\in T^{[\varepsilon]}(x)$; 

(iv) $T_1^{[\varepsilon_1]}(x)+T_2^{[\varepsilon_2]}(x)\subseteq (T_1+T_2)^{[\varepsilon_1+\varepsilon_2]}(x)$ for all $x\in{\cal H}$ and 
$\varepsilon_1,\varepsilon_2\geq 0$; 

(v) $Az\in A^{[\varepsilon]}(x)$, for all $x,z\in{\cal H}$, where $\varepsilon=\frac{\|x-z\|^2}{4\gamma}$. 
\end{proposition}

We close this section by presenting two convergence results which will be crucial for the proof of the main results
in the next section.

\begin{lemma}\label{ext-fejer1} (see \cite{alvarez-attouch2001, alvarez2000, alvarez2004}) Let $(\varphi_k)_{k\in\N}, (\delta_k)_{k\in\N}$ and $(\alpha_k)_{k\in \N}$ be sequences in
$[0,+\infty)$ such that $\varphi_{k+1}\leq\varphi_k+\alpha_k(\varphi_k-\varphi_{k-1})+\delta_k$
for all $k \geq 1$, $\sum_{k\in \N}\delta_k< + \infty$ and there exists a real number $\alpha$ with
$0\leq\alpha_k\leq\alpha<1$ for all $k\in\N$. Then the following hold: \begin{itemize}\item[(i)] $\sum_{k \geq 1}[\varphi_k-\varphi_{k-1}]_+< + \infty$, where
$[t]_+=\max\{t,0\}$; \item[(ii)] there exists $\varphi^*\in[0,+\infty)$ such that $\lim_{k\rightarrow+\infty}\varphi_k=\varphi^*$.\end{itemize}
\end{lemma}

\begin{lemma}\label{opial} (Opial, see for example \cite{bauschke-book}) Let $C$ be a nonempty set of ${\cal H}$ and $(x^k)_{k\in\N}$ be a sequence in ${\cal H}$ such that
the following two conditions hold: \begin{itemize}\item[(a)] for every $x\in C$, $\lim_{k\rightarrow + \infty}\|x^k-x\|$ exists;
\item[(b)] every sequential weak cluster point of $(x^k)_{k\in\N}$ is in $C$;\end{itemize}
Then $(x^k)_{k\in\N}$ converges weakly to a point in $C$.
\end{lemma}

\section{An inertial hybrid proximal point algorithm}\label{sec2}

This section is dedicated to the formulation of an inertial hybrid proximal-extragradient algorithm and the convergence analysis of it. The iterative scheme we propose for 
finding the zeros of a given maximally monotone operator $T:{\cal H}\rightrightarrows{\cal H}$  has the following form. 

\begin{algorithm}\label{alg-inertial-hyb} Choose $x^0,x^1,x^2,y^0,y^1,v^1\in{\cal H}$, $\alpha,\sigma\geq 0$, $\ul c>0$, $(c_k)_{k\geq 1}$ and $(\alpha_k)_{k\geq 1}$ 
such that $$c_k\geq \ul c>0 \quad \forall k \geq 1,$$ 
$$0\leq\alpha_k\leq\alpha \quad \forall k\geq 1$$ and
\begin{equation}\label{param}\alpha(5 +4\sigma^2)+\sigma^2 < 1.\end{equation}
For every $k\geq 2$ consider the following iterative scheme: 
\begin{itemize} \item[(i)] for some 
$\varepsilon_k\geq 0$, choose $v^k\in T^{[\varepsilon_k]}(y^k)$ such that
\begin{align*}
& \hspace{-1cm} 2c_k\varepsilon_k+\|c_kv^k+y^k-x^k-\alpha_k(x^k-x^{k-1})\|^2 +\\
& \hspace{-1cm} 4 \alpha_k \|c_{k-1}v^{k-1}+y^{k-1}-x^{k-1}-\alpha_{k-1}(x^{k-1}-x^{k-2})\|^2 \\
& \hspace{7cm} \leq \sigma^2\|y^k-x^k\|^2+4 \alpha_k \sigma^2\|y^{k-1}-x^{k-1}\|^2;
\end{align*}
\item[(ii)] define $x^{k+1}=x^k+\alpha_k(x^k-x^{k-1})-c_kv^k$.\end{itemize}
\end{algorithm}

\subsection{Relation to other splitting algorithms from the literature}

Before analyzing the convergence of the above algorithm, we show that several algorithms from the literature can be embedded in the setting of this inertial hybrid scheme, 
by following some techniques from \cite{sv}.

(i) The {\it hybrid proximal-extragradient algorithm} (see \cite{sol-sv1999}) presented in Algorithm \ref{alg-hyb-classic} follows by taking $\alpha=0$, which enforces $\alpha_k=0$ for all $k\geq 1$. 

(ii) The {\it inertial proximal point algorithm} (see \cite{alvarez-attouch2001}) for finding the zeros of $T$ reads:
\begin{equation}\label{inertial-prox-alg}x^{k+1}=J_{c_kT}\big(x^k+\alpha_k(x^k-x^{k-1})\big) \ \forall k\geq 1,\end{equation}
where $0\leq\alpha_k\leq\alpha < \frac{1}{5}$ for every $k \geq 1$.

By taking in Algorithm \ref{alg-inertial-hyb} $\sigma=0$, we obtain for every $k\geq 2$ that $\varepsilon_k=0$ and $y^k=x^k+\alpha_k(x^k-x^{k-1})-c_kv^k=x^{k+1}$. 
Since $v^k\in T^{[0]}(y^k)$ we derive from Proposition \ref{prop-enl}(i) 
$$x^k-x^{k+1}+\alpha_k(x^k-x^{k-1})=c_kv^k\in c_kTy^k=c_kTx^{k+1},$$
which, by the definition of the resolvent, is nothing else than the iterative scheme \eqref{inertial-prox-alg} starting with $k=2$. 

(iii) The {\it inertial forward-backward algorithm} for finding the zeros of $T:=A+B$, where $A:{\cal H}\rightarrow{\cal H}$ is a $\gamma$-cocoercive operator
with $\gamma >0$ and $B:{\cal H}\rightrightarrows{\cal H}$ is a maximally monotone operator reads in the error-free case (see \cite{moudafi-oliny2003}): 
\begin{equation}\label{inertial-fb} x^{k+1}=J_{c_k B}\left(x^k-c_kAx^k+\alpha_k(x^k-x^{k-1})\right) \ \forall k\geq 1, 
\end{equation}
where for $\alpha, \sigma \geq 0$ fulfilling \eqref{param} it is assumed that $0\leq\alpha_k\leq\alpha$ and  $0<\ul c\leq c_k\leq 2\gamma \sigma^2$ for every $k\geq 1$. 

Considering $(x^k)_{k \in \N}$ the sequence generated by  \eqref{inertial-fb}, for every $k\geq 1$ we define: 
\begin{align*}
v^k & =\frac{1}{c_k}(x^k-x^{k+1})+\frac{\alpha_k}{c_k}(x^k-x^{k-1})\\
y^k & =x^{k+1}\\
\varepsilon_k^1 & =\frac{\|x^{k+1}-x^k\|^2}{4\gamma}\\ 
\varepsilon_k^2 & =\frac{\alpha_k}{\gamma}\|x^k-x^{k-1}\|^2\\
\varepsilon_k & =\varepsilon_k^1+\varepsilon_k^2.
\end{align*}

Let $k \geq 2$ be fixed. By the choice of $v^k$, the equality (ii) in Algorithm \ref{alg-inertial-hyb} is obviously verified. Moreover, from \eqref{inertial-fb} we derive that $v^k\in Ax^k+Bx^{k+1}$. 

From (i)-(ii) and (iv)-(v) in Proposition \ref{prop-enl} we get
\begin{align*}
v^k\in A^{[\varepsilon_k^1]}(x^{k+1})+Bx^{k+1}\subseteq &  A^{[\varepsilon_k^1]}(x^{k+1})+B^{[\varepsilon_k^2]}(x^{k+1})\subseteq\\ 
& (A+B)^{[\varepsilon_k]}(x^{k+1})=T^{[\varepsilon_k]}(x^{k+1})=T^{[\varepsilon_k]}(y^k).
\end{align*}

Finally, we show that the inequality in Algorithm \ref{alg-inertial-hyb}(i) holds. By the choices we met we have $c_lv^l+y^l-x^l-\alpha_l(x^l-x^{l-1})=0$ for all $l\geq 1$, hence 
\begin{align*}
& 2c_k\varepsilon_k+\|c_kv^k+y^k-x^k-\alpha_k(x^k-x^{k-1})\|^2+\\
&4\alpha_k\|c_{k-1}v^{k-1}+y^{k-1}-x^{k-1}-\alpha_{k-1}(x^{k-1}-x^{k-2})\|^2 = 2c_k\varepsilon_k\leq 4\gamma \sigma^2\varepsilon_k=\\
& \sigma^2\|x^{k+1}-x^k\|^2+4\alpha_k\sigma^2\|x^k-x^{k-1}\|^2 = \sigma^2\|y^k-x^k\|^2+4\alpha_k\sigma^2\|y^{k-1}-x^{k-1}\|^2.
\end{align*}

(iv) Finally, we consider the {\it inertial forward-backward-forward algorithm} (see \cite{b-c-inertial}) for finding the zeros of $T:=A+B$, where $A:{\cal H}\rightarrow{\cal H}$ is  a monotone and $\beta$-Lipschitz continuous operator 
with $\beta \geq 0$ and $B:{\cal H}\rightrightarrows{\cal H}$ is a maximally monotone operator. According to \cite{b-c-inertial} (see also \cite[Remark 6]{b-c-inertial}) this has the following iterative scheme: 
$$\left\{
\begin{array}{ll}
y^k=J_{c_k B}[x^k-c_kAx^k+\alpha_k(x^k-x^{k-1})]\\
x^{k+1}=y^k+c_k(Ax^k-Ay^k),
\end{array}\right. \quad \forall k\geq 1,$$
for $\alpha, \sigma \geq 0$ fulfilling \eqref{param} and $\ol \sigma >0$ chosen such that $\frac{1-5\alpha - \sigma^2(4\alpha + 1)}{2(\sigma^2 + 1)} \leq \ol \sigma < \frac{1-5\alpha}{2}$, it is assumed that $0\leq\alpha_k\leq\alpha$ and  
$0<\ul c\leq c_k\leq\frac{1}{\beta}\sqrt{\frac{1-5\alpha-2\ol \sigma}{4\alpha+2\ol \sigma+1}}$ for every $k\geq 1$. 

Consider the sequences $(x^k)_{k \in \N}$ and $(y^k)_{k \in \N}$ generated by this algorithm. For every $k\geq 1$ we define: 
\begin{align*}
\varepsilon_k & =0\\
b^k & =\frac{1}{c_k}(x^k-y^k)-Ax^k+\frac{\alpha_k}{c_k}(x^k-x^{k-1}),\\
v^k & =Ay^k+b^k\\
r^k & =c_kv^k+y^k-x^k-\alpha_k(x^k-x^{k-1}). 
\end{align*}
Let $k \geq 2$ be fixed. The definition of the resolvent yields $b^k\in By^k$, hence $v^k\in Ty^k$. Further, from the 
definition of $b^k$ we get 
\begin{equation*} r^k=c_kAy^k+c_kb^k+y^k-x^k-\alpha_k(x^k-x^{k-1})=c_k(Ay^k-Ax^k),\end{equation*}
hence $$x^{k+1}=y^k+c_k(Ax^k-Ay^k)=y^k-r^k=x^k+\alpha_k(x^k-x^{k-1})-c_kv^k$$
and the update rule in Algorithm \ref{alg-inertial-hyb}(ii) is verified. In what concerns the inequality in Algorithm \ref{alg-inertial-hyb}(i), we notice first that
$\sqrt{\frac{1-5\alpha-2\ol \sigma}{4\alpha+2\ol \sigma+1}}\leq \sigma<1$. By using  the fact that $A$ is $\beta$-Lipschitz, we get
\begin{align*}
& 2c_k\varepsilon_k+\|c_kv^k+y^k-x^k-\alpha_k(x^k-x^{k-1})\|^2+\\
& 4\alpha_k\|c_{k-1}v^{k-1}+y^{k-1}-x^{k-1}-\alpha_{k-1}(x^{k-1}-x^{k-2})\|^2 = \|r^k\|^2+4\alpha_k\|r^{k-1}\|^2=\\
& \|c_k(Ay^k-Ax^k)\|^2+4\alpha_k\|c_{k-1}(Ay^{k-1}-Ax^{k-1})\|^2 \leq\\
& c_k^2\beta^2\|y^k-x^k\|^2+4\alpha_k c_{k-1}^2\beta^2\|y^{k-1}-x^{k-1}\|^2 \leq \\
& \sigma^2 \|y^k-x^k\|^2+4\alpha_k \sigma^2\|y^{k-1}-x^{k-1}\|^2.
\end{align*}

\subsection{Convergence analysis}

In this subsection we prove the convergence of the proposed inertial hybrid proximal-extragradient algorithm.

\begin{theorem}\label{th-inertial-hybrid} Let $T:{\cal H}\rightrightarrows {\cal H}$ be a maximally monotone operator such that $\zer T\neq\emptyset$. Consider the 
sequences generated by Algorithm \ref{alg-inertial-hyb}, where $(\alpha_k)_{k\geq 1}$ is supposed to be nondecreasing and we either take $\alpha_1=0$ or $x^1=x^0$. Then the following statements are true: 

(i) $\sum_{k\in\N}\|x^{k+1}-x^k\|^2<+\infty$, $\sum_{k\in\N}\|x^k-y^k\|^2<+\infty$, $\sum_{k\geq 1}\|v^k\|^2<+\infty$ and $\sum_{k\geq 2}\varepsilon_k<+\infty$;

(ii) $(x^k)_{k\in\N}$ converges weakly to an element in $\zer T$.  
\end{theorem}

\begin{proof} We fix an element $z\in \zer T$ and $k\geq 1$ and  make the following notations 
\begin{equation*}r^k=c_kv^k+y^k-x^k-\alpha_k(x^k-x^{k-1}) \ \mbox{and} \ \varphi_k=\frac{1}{2}\|x^k-z\|^2.
\end{equation*}
By the update rule in Algorithm \ref{alg-inertial-hyb}(ii) one obviously has
\begin{equation}\label{r2}r^k=y^k-x^{k+1}.\end{equation}
Since $v^k\in T^{[\varepsilon_k]}(y^k)$ and $0\in Tz$, the definition of the enlargement yields the inequality 
$$\langle v^k,y^k-z\rangle\geq -\varepsilon_k.$$

Multiplying it with $-c_k$ and taking into account Algorithm \ref{alg-inertial-hyb}(ii) and \eqref{r2} we derive 
\begin{equation}\label{ineq1}\langle x^{k+1}-x^k-\alpha_k(x^k-x^{k-1}),x^{k+1}-z\rangle\leq c_k\varepsilon_k+\langle c_kv^k,r^k\rangle.\end{equation}

Let us take now a look at the left-hand side of the above inequality. We have 
$$\langle x^{k+1}-x^k-\alpha_k(x^k-x^{k-1}),x^{k+1}-z\rangle
$$$$=
\langle x^{k+1}-x^k,x^{k+1}-z\rangle-\alpha_k(\langle x^k-x^{k-1},x^k-z\rangle+\langle x^k-x^{k-1},x^{k+1}-x^k\rangle)$$
$$=\frac{1}{2}\|x^{k+1}-x^k\|^2+\varphi_{k+1}-\varphi_k-\alpha_k\left(\frac{1}{2}\|x^k-x^{k-1}\|^2+\varphi_k-\varphi_{k-1}+\langle x^k-x^{k-1},x^{k+1}-x^k\rangle\right)$$
$$=\varphi_{k+1}-\varphi_k-\alpha_k(\varphi_k-\varphi_{k-1})+\frac{1}{2}\|x^{k+1}-x^k\|^2-\frac{\alpha_k}{2}\|x^k-x^{k-1}\|^2-\alpha_k\langle x^k-x^{k-1},x^{k+1}-x^k\rangle.$$
The term $\langle c_kv^k,r^k\rangle$ in the right-hand side of \eqref{ineq1} can be written as
\begin{align*}
\langle c_kv^k,r^k\rangle & =\langle x^k+\alpha_k(x^k-x^{k-1})-x^{k+1},y^k-x^{k+1}\rangle\\
& =\langle x^k-x^{k+1},y^k-x^{k+1}\rangle+\alpha_k\langle x^k-x^{k-1},y^k-x^{k+1}\rangle\\
& =\frac{1}{2}\|x^{k+1}-x^k\|^2+\frac{1}{2}\|r^k\|^2-\frac{1}{2}\|x^k-y^k\|^2+\alpha_k\langle x^k-x^{k-1},y^k-x^{k+1}\rangle.
\end{align*}

Consequently, \eqref{ineq1} can be equivalently written as
\begin{align}\label{ineq2} \varphi_{k+1}-\varphi_k-\alpha_k(\varphi_k-\varphi_{k-1})\leq \ &\alpha_k\langle x^k-x^{k-1},y^k-x^k\rangle+\frac{\alpha_k}{2}\|x^k-x^{k-1}\|^2\nonumber\\
&+c_k\varepsilon_k+\frac{1}{2}\|r^k\|^2-\frac{1}{2}\|x^k-y^k\|^2.
\end{align}
Further, 
\begin{align*}
\alpha_k\langle x^k-x^{k-1},y^k-x^k\rangle  + & \frac{\alpha_k}{2}\|x^k-x^{k-1}\|^2 \leq \alpha_k\|x^k-x^{k-1}\|^2+\frac{\alpha_k}{2}\|y^k-x^k\|^2\\
& \leq 2 \alpha_k (\|r^{k-1}\|^2+\|x^{k-1}-y^{k-1}\|^2)+\frac{\alpha_k}{2}\|y^k-x^k\|^2
\end{align*}
and from \eqref{ineq2} we obtain 
\begin{align}\label{ineq3} \varphi_{k+1}-\varphi_k-\alpha_k(\varphi_k-\varphi_{k-1})\leq \ & \frac{\alpha_k-1}{2}\|x^k-y^k\|^2+
2\alpha_k\|x^{k-1}-y^{k-1}\|^2\nonumber\\
 &+c_k\varepsilon_k+\frac{1}{2}\|r^k\|^2+2\alpha_k\|r^{k-1}\|^2.
\end{align}

On the other hand, the inequality in Algorithm \ref{alg-inertial-hyb}(ii) yields 
\begin{equation}\label{ineq-alg}c_k\varepsilon_k+\frac{1}{2}\|r^k\|^2+2\alpha_k \|r^{k-1}\|^2\leq \frac{\sigma^2}{2}\|y^k-x^k\|^2+2\alpha_k \sigma^2\|y^{k-1}-x^{k-1}\|^2,\end{equation}
hence from \eqref{ineq3} we get 
\begin{equation}\label{ineq4} \varphi_{k+1}-\varphi_k-\alpha_k(\varphi_k-\varphi_{k-1})\leq
-\frac{1 - \alpha_k - \sigma^2}{2} \|x^k-y^k\|^2+2\alpha_k(1+\sigma^2)\|x^{k-1}-y^{k-1}\|^2.
\end{equation}

(i) For the proof of this statement we are going to use some techniques from \cite{alvarez-attouch2001}. We define the sequence 
$$\mu_k:=\varphi_k-\alpha_k\varphi_{k-1}+2\alpha_k(1+\sigma^2)\|x^{k-1}-y^{k-1}\|^2 \ \forall k \geq 1.$$
Using the monotonicity of $(\alpha_k)_{k \geq 1}$ and the fact that $\varphi_k\geq 0$ for every $k\geq 1$, we get
\begin{align*}
 \mu_{k+1}-\mu_k\leq \ &\varphi_{k+1}-\varphi_k-\alpha_k(\varphi_k-\varphi_{k-1})\\
& +2\alpha_{k+1}(1+\sigma^2)\|x^k-y^k\|^2-2\alpha_k(1+\sigma^2)\|x^{k-1}-y^{k-1}\|^2,
\end{align*}
which gives by \eqref{ineq4}
\begin{equation}\label{ineq5}
\mu_{k+1}-\mu_k\leq-\left(\frac{1 - \alpha_k - \sigma^2}{2}-2\alpha_{k+1}(1+\sigma^2)\right)\|x^k-y^k\|^2 \ \forall k \geq 1.
\end{equation}

The upper bound requested for $(\alpha_k)_{k\geq 1}$ and \eqref{param} shows the inequality 
\begin{equation}\label{claim}\frac{1 - \alpha_k - \sigma^2}{2}-2\alpha_{k+1}(1+\sigma^2) \geq \frac{1-\alpha(5 + 4 \sigma^2) - \sigma^2}{2} > 0 \ \forall k \geq 1,\end{equation}
thus 
\begin{equation}\label{ineq6}\mu_{k+1}-\mu_k\leq-\frac{1-\alpha(5 + 4 \sigma^2) - \sigma^2}{2}\|x^k-y^k\|^2 \ \forall k\geq 1.\end{equation}

The sequence $(\mu_k)_{k \geq 1}$ is nonincreasing and the bound for $(\alpha_k)_{k \geq 1}$ delivers 
\begin{equation}\label{ineq7}
-\alpha\varphi_{k-1}\leq\varphi_k-\alpha\varphi_{k-1}\leq \mu_k\leq\mu_1 \ \forall k \geq 1.
\end{equation}

We obtain
$$\varphi_k\leq \alpha^k\varphi_0+\mu_1\sum_{i=0}^{k-1}\alpha^i\leq \alpha^k\varphi_0+\frac{\mu_1}{1-\alpha} \ \forall k \geq 1,$$
where we notice that $\mu_1\geq 0$. Indeed, in case $\alpha_1=0$ one has $\mu_1=\varphi_1\geq 0$, while in 
the case $x^1=x^0$ we have $\varphi_1=\varphi_0$ and 
$\mu_1=\varphi_1-\alpha_1\varphi_0+2\alpha_1(1+\sigma^2)\|x^0-y^0\|^2\geq (1-\alpha)\varphi_0+2\alpha_1(1+\sigma^2)\|x^0-y^0\|^2\geq 0$ due to \eqref{param}. 
Combining \eqref{ineq6} and \eqref{ineq7} we get for every $n \geq 1$
$$\frac{1-\alpha(5 + 4 \sigma^2) - \sigma^2}{2} \sum_{k=1}^{n}\|x^k-y^k\|^2\leq \mu_1-\mu_{n+1}\leq \mu_1+\alpha\varphi_n\leq \alpha^{n+1}\varphi_0+\frac{\mu_1}{1-\alpha},$$ which
shows that $\sum_{k\in\N}\|x^k-y^k\|^2<+\infty$.

The fact that $\sum_{k\geq 2}\varepsilon_k<+\infty$ follows now from \eqref{ineq-alg}, since $c_k\geq \ul c$ and $\alpha_k\leq\alpha$ for every $k\geq 1$. 
Notice that from \eqref{ineq-alg} we deduce also that $\sum_{k\geq 1}\|r^k\|^2<+\infty$. Further, from \eqref{r2} we have 
$\sum_{k=1}^n\|x^{k+1}-x^k\|^2\leq 2\left(\sum_{k=1}^n\|r^k\|^2+\sum_{k=1}^n\|y^k-x^k\|^2\right)$ for every $n\geq 1$, hence $\sum_{k\in\N}\|x^{k+1}-x^k\|^2<+\infty$.
Finally, from Algorithm \ref{alg-inertial-hyb}(ii) we derive that $\sum_{k\geq 1}\|v^k\|^2<+\infty$.

(ii) In order to prove this statement we are going to use Lemma \ref{opial}. We shown above that for an arbitrary $z\in\zer T$ the inequality \eqref{ineq4} is true. By \eqref{claim} we get 
\begin{equation}\label{ineq-} \varphi_{k+1}-\varphi_k-\alpha_k(\varphi_k-\varphi_{k-1})\leq 2\alpha_k(1+\sigma^2)\|x^{k-1}-y^{k-1}\|^2 \ \forall k\geq 1
\end{equation} and from Lemma \ref{ext-fejer1} and  
part (i) it follows that $\lim_{k\rightarrow+\infty}\|x^k-z\|$ exists. On the other hand, let $x$ be a sequential weak cluster point of
$(x^k)_{k\in\N}$, that is, this sequence has a subsequence $(x^{k_n})_{n \in \N}$ fulfilling $x^{k_n}\rightharpoonup x$ as $n\rightarrow+\infty$. Since $x^k-y^k\rightarrow 0$ 
as $k\rightarrow+\infty$, 
we get $y^{k_n}\rightharpoonup x$ as $n\rightarrow+\infty$. Further, we have $v^{k_n}\in T^{[\varepsilon_{k_n}]}(y^{k_n})$, 
$v^{k_n}\rightarrow 0$ and $\varepsilon^{k_n}\rightarrow 0$ as $n\rightarrow+\infty$, hence from Proposition \ref{prop-enl}(iii) and (i) we deduce $0\in Tx$, 
thus $x\in\zer T$. By Lemma \ref{opial}, $(x^k)_{k\in\N}$ converges weakly to an element in $\zer T$.
\end{proof}

\begin{remark} By arguing in a similar manner as for Algorithm \ref{alg-inertial-hyb}, one can prove the convergence of the following inertial-type hybrid proximal-extragradient
algorithm, as well:

\begin{algorithm}\label{alg-inertial-hyb-var} Choose $x^0,x^1,x^2,y^0,y^1,v^1\in{\cal H}$, $\alpha,\sigma\geq 0$, $\ul c>0$, $(c_k)_{k\geq 1}$ and $(\alpha_k)_{k\geq 1}$ 
such that $$c_k\geq \ul c>0 \quad \forall k \geq 1,$$ 
$$0\leq\alpha_k\leq\alpha \quad \forall k\geq 1$$ and
\begin{equation}\label{param-var} 5\alpha + \sigma^2 < 1.\end{equation}
For every $k\geq 2$ consider the following iterative scheme: 
\begin{itemize} \item[(i)] for some 
$\varepsilon_k\geq 0$, choose $v^k\in T^{[\varepsilon_k]}(y^k)$ such that
\begin{align*}
& 2c_k\varepsilon_k+\|c_kv^k+y^k-x^k-\alpha_k(x^k-x^{k-1})\|^2 +\\
& 4 \alpha_k \|c_{k-1}v^{k-1}+y^{k-1}-x^{k-1}-\alpha_{k-1}(x^{k-1}-x^{k-2})\|^2 \leq \sigma^2\|y^k-x^k\|^2;
\end{align*}
\item[(ii)] define $x^{k+1}=x^k+\alpha_k(x^k-x^{k-1})-c_kv^k$.\end{itemize}
\end{algorithm}
The differences between the two iterative scheme are in the relations \eqref{param} and \eqref{param-var} and the two inequalities in the statements (i), respectively. 
The \textit{hybrid proximal-extragradient}, the \textit{inertial proximal point} and the \textit{inertial forward-backward} algorithms can be rediscovered as particular instances
of this iterative scheme, too (notice that for the latter on needs to take $\varepsilon_k^2=0, k \geq 1$). However, the \textit{inertial forward-backward-forward algorithm} 
cannot be embedded in Algorithm \ref{alg-inertial-hyb-var} and this is why we opted in this paper for the inertial version provided in Algorithm \ref{alg-inertial-hyb}, despite its more complex formulation.
\end{remark}

\end{document}